\documentclass{amsart}
\usepackage{xcolor}
\usepackage{hyperref}
\usepackage[utf8]{inputenc}
\usepackage{amsaddr}
\usepackage[all]{xy}
\usepackage{amssymb}
\usepackage{amsmath}
\usepackage{enumerate}

\newtheorem{theorem}{Theorem}[section]
\newtheorem{conjecture}[theorem]{Conjecture}
\newtheorem{proposition}[theorem]{Proposition}
\newtheorem{definition}{Definition}

\newtheorem{lemma}{Lemma}
\newtheorem{remark}{Remark}

\newtheorem{ithm}{Theorem}[section]

\newtheorem{iconj}[ithm]{Conjecture}

\DeclareMathOperator{\Ricci}{Ric}

\def\cal#1{\mathcal{#1}}

\def\lie#1{\mathfrak{#1}}

\usepackage[all]{xy}

\newcommand{\h}{\frac{1}{2}}

\newcommand{\ga}{\textsl{g}}
\newcommand{\aga}{\textsl{h}}
\makeatletter
\renewcommand{\email}[2][]{%
	\ifx\emails\@empty\relax\else{\g@addto@macro\emails{,\space}}\fi%
	\@ifnotempty{#1}{\g@addto@macro\emails{\textrm{(#1)}\space}}%
	\g@addto@macro\emails{#2}%
}
\makeatother

\newtheorem{claim}{Claim}

\title[Cheeger deformation on fiber bundles]{The concept of Cheeger deformations on fiber bundles with compact structure group}
\author{Leonardo F. Cavenaghi$^{\dagger}$}
\address{Instituto de Matemática, Estatística e Computação Cinetífica -- Unicamp, Rua Sérgio Buarque de Holanda, 651, 13083-859, Campinas, SP, Brazil}
\email[$\dagger$]{leonardofcavenaghi@gmail.com}
\author{Lino Grama$^{\ast}$}
\address{Instituto de Matemática, Estatística e Computação Cinetífica -- Unicamp, Rua Sérgio Buarque de Holanda, 651, 13083-859, Campinas, SP, Brazil}
\email[$\ast$]{lino@ime.unicamp.br}
\author{Llohann D. Sperança$^{\star}$}
\address{Instituto de Ciência e Tecnologia -- Unifesp, Avenida Cesare Mansueto Giulio Lattes, 1201, 12247-014, São José dos Campos, SP, Brazil}
\email[$\star$]{lsperanca@gmail.com}
\begin{document}
	

	\begin{abstract}
The purpose of this paper is two-fold: we systematically introduce the notion of Cheeger deformations on fiber bundles with compact structure groups, and recover in a very simple and unified fashion several results that either already appear in the literature or are known by experts, though are not explicitly written elsewhere. We re-prove: Schwachh\"ofer--Tuschmann Theorem on bi-quotients, many results due to Fukaya and Yamaguchi, as well as, naturally extend the work of Searle--Solórzano--Wilhelm on regularization properties of Cheeger deformations, among others. In this sense, this paper should be understood as a survey intended to demonstrate the power of Cheeger deformations. Even though some of the results here appearing may not be known as stated in the presented form, they were already expected, being our contribution to the standardization and spread of the technique via a unique language.\end{abstract}
	
	\maketitle
	
	\section*{Conflict of interest}
	On behalf of all authors, the corresponding author states that there is no conflict of interest.
	
	\section{Introduction}
	
The metric deformation known as \emph{Cheeger deformations} was firstly introduced on \cite{cheeger}. Its main goal was to produce metrics with non-negative sectional curvature on manifolds with symmetries. Since then, Cheeger deformations were used in \cite{gz,grove2011exotic} to produce new examples of manifolds with non-negative and positive sectional curvature; in \cite{schwachhofer2004metrics,schwachhofer2009homogeneous,kerr2013nonnegatively} to study curvature properties on homogeneous spaces such as biquotients, on
	\cite{searle2015lift} to lift positive Ricci curvature from a metric quotient $M/G$ to $M$, and in \cite{gz,SperancaCavenaghiPublished, CavenaghiSperanca+2022+95+104} to provide examples of manifolds with non-negative sectional and/or positive Ricci curvatures. Other interesting results along the same lines are in \cite{schwachhofer2004metrics,tusch,grove,WB}. Here we introduce an analogous metric deformation, defined on a specific class of metrics on fiber bundles with compact structure groups, naturally supported on Cheeger deformations. Throughout the manuscript we provide other references, including results on the existence metrics of \emph{almost non-negative sectional curvature}.
	
	 Recall that from any fiber bundle $F \hookrightarrow M \rightarrow B$ with compact structure group $G$ can be decoupled a principal $G$-bundle $\cal P \rightarrow B$ and a manifold $F$ with an effective $G$-action. Here, $M$ can be recovered via a submersion $\overline \pi : \cal P \times F \rightarrow M$ with fiber $G$. The idea of our deformation consists of inducing a one parameter family of metrics on $M$ via $\overline \pi$ after making Cheeger deformation on $\cal P$. That is, if $\ga$ is a Riemannian metric on $\cal P$ for which $G$ acts via isometries, given any $G$-invariant metric $\ga_F$ on $F$ we look to the metric $\aga_t$ on $M$ obtained from $\ga_t + \ga_F$ in $\cal P \times F$, see Definition \ref{defn} for further details.
	
	All the long we mostly follow the approach in \cite{mutterz} and \cite{Muter}, introducing useful tensors to standardize the analysis of this deformation, such as nowadays well established the basics on Cheeger deformations. As a very useful formula we shall obtain:
	
	Let $\textsl h$ obtained via $\overline \pi : (\cal P \times F, \ga + \ga_F) \rightarrow (M,\aga)$ and let $\textsl g_t$ be a Cheeger deformation of $\ga$. Then, for every pair $\tilde X = X + X_F + U^*$, $\tilde Y = Y + Y_F + V^*$ of tangent vectors to $M$, appropriately decomposed, it holds that 
		\begin{equation}\label{eq:mutergeneral}
		\tilde \kappa_{t}(\tilde X, \tilde Y) = \kappa_t(X{+}U^{\vee},Y{+}V^{\vee}) {+} K_{\textsl g_F}(X_F - (P_F^{-1}PU)^*, Y_F - (P_F^{-1}PV)^*) + \tilde z_t(\tilde X,\tilde Y),
		\end{equation}
		where $\tilde \kappa_t$ is the unreduced sectional curvature of the metric $\aga_t$ computed in an appropriate reparametrization of the plane $\tilde X\wedge \tilde Y$, $\kappa_t$ is the unreduced sectional curvature of $\ga_t$ and $K_{\ga_F}$ is the unreduced sectional curvature of $\ga_F$. Moreover, $\tilde z_t$ is a non-negative term.
		
		We stress it out that it is not of the author's knowledge whether equation \eqref{eq:mutergeneral} already appears elsewhere in such a general manner. However, when collapsing the fiber $F$ to a single point, it naturally yields to the well known expression of the sectional curvature of a Cheeger deformation computed at some reparameterized planes, see \cite[Proposition 1.3, p.2]{mutterz} or equation \eqref{eq:curvaturaseccional}.

Taking advantage of equation \eqref{eq:mutergeneral}, we re-prove in a very general picture results on almost non-negative sectional curvature appearing in \cite{fukaya-yamaguchi}. Such results as stated were either already known (see \cite{mutterz}) or expected to be true, though not explicitly written elsewhere.
	
	\begin{theorem}[Fukaya--Yamaguchi]\label{thm:mainintro}
Let $F \hookrightarrow M \to B$ be a bundle with compact structure group $G$, fiber $F$ and base $B$. Assume that $M$ is an associate bundle to $\pi : (\cal P,\textsl g) \to B$ such that:
\begin{enumerate}
\item $K_{\textsl g}\geq 0$;
\item $F$ has a $G$-invariant metric $\textsl g_F$ of non-negative sectional curvature.
\end{enumerate}
Then $M$ admits a sequence of Riemannian metrics
$\{\textsl g_n\}$ such that $\mathrm{sec}_{\textsl g_n} \geq -\frac{1}{n},$ $\mathrm{diam}~(M,\textsl g_n) \leq \frac{1}{n}.$
\end{theorem}

A Fukaya--Yamaguchi type result on the existence of almost non-negative Ricci curvature, namely:

\begin{theorem}\label{thm:previagromovintro}
Let $F\hookrightarrow M \rightarrow B$ be a fiber bundle with compact structure group $G$ and total space $M$. Also assume that $F$ carries a metric $\textsl g_F$ of non-negative Ricci curvature and $B$ carries a metric $\textsl g_{\epsilon}$ with $\Ricci(\textsl g_{\epsilon}) \geq -\epsilon^2$. Then $M$ carries a metric $\textsl h_{\epsilon}$ with $\Ricci(\textsl h_{\epsilon}) \geq -\epsilon^2$.
\end{theorem}

Both Theorems \ref{thm:mainintro} and \ref{thm:previagromovintro} should follow from the computations in \cite{schwachhofer2004metrics}, though these follow very directly from our techniques. We also reinforce that Theorem \ref{thm:previagromovintro} was first conjectured to be true in \cite[Conjecture 0.14, p.257]{fukaya-yamaguchi}, see also \cite{guowei,colding, gyun}. Notably as well is the fact that this kind of result is of interest in the field of Metric Geometry, though we do not touch this area here, being the above-mentioned theorem proofs of concept to the deformation here developed. Other very useful references related to these subjects are: \cite{WILKING2000129,https://doi.org/10.48550/arxiv.1105.5955}.

All the analyses coming out from equation \eqref{eq:mutergeneral} allows us to recover in a very simple fashion classical results in bi-quotients, such as:

\begin{theorem}[Schwachh{\"o}fer--Tuschmann]\label{thm:biqnintro}
Any bi-quotient $G//K$ from a compact Lie group $G$ admits a metric with positive Ricci curvature and almost non-negative sectional curvature simultaneously if, and only if, $G//K$ has finite fundamental group.
\end{theorem}

Finally, we obtained two further applications. Recall for instance that in \cite{solorzano} Searle--Solórzano--Wilhelm show that Cheeger deformations work as a strong regularization process: appropriate scaling of the family of metrics on Cheeger deformations imply $C^p$-convergence, for any $p\geq 0$ a priori fixed, to metrics with totally geodesic fibers. We apply this idea here to prove:

	\begin{theorem}\label{thm:analogousSolorzanointro}
	Let $\pi : F\hookrightarrow M \rightarrow B$ be a fiber bundle with compact total space and compact structure group $G$. Assume that $\aga$ is a Riemannian submersion metric on $M$ obtained via the submersion $\overline \pi : \cal (P\times F,\textsl g+\ga_F) \rightarrow M$, where $\cal P$ is the associated principal bundle to $\pi$ and $\textsl g, \ga_F$ are, respectively, $G$-invariant metrics on $\cal P$ and $F$. Then, for any integer $p\geq 0$, after an appropriate re-scaling the fibers of $\pi$, the metric deformation $\aga_t$ (Definition \ref{defn}, section \ref{sec:principal}), converges in the $C^p$-topology to a Riemannian submersion metric with totally geodesic fibers. 
	\end{theorem}
	
	Theorem \ref{thm:analogousSolorzanointro} was already expected to be true but more importantly, this regularization nature of Cheeger deformations already appears in the proof of all the mentioned results in this intro. As it is also clear, the sectional curvature formulae \eqref{eq:mutergeneral} may lead to new conjectures in which concerns the existence of metrics with positive sectional curvature on the total space of some fiber bundles.
	
	Recall, for instance, the fiber dimension Petersen--Wilhelm conjecture:
		
	\begin{iconj}[Petersen--Wilhelm Fiber dimension conjecture]\label{conj:wilhelmintro} If $F\hookrightarrow M\rightarrow B$ is a Riemannian submersion from a positively curved closed manifold $M$, then $$ \dim F<\dim B.$$
	\end{iconj}
	
	In Section \ref{sec:petwill} we make some comments on this conjecture in the case of fiber bundles with the structure group being $S^3, SO(3)$. More precisely, we conjecture:
	
	\begin{iconj}[Principal bundle Strong Petersen--Wilhelm conjecture]\label{conj:principalintro}
Any $S^3, SO(3)$ principal bundle over a positively curved manifold admits a metric with positive sectional curvature if, and only if, such a submersion is fat.
\end{iconj}

Assuming the validity of Conjecture \ref{conj:principalintro} it shall be straightforward to check that: Any $S^2\hookrightarrow M \rightarrow B$ fat bundle with structure group $SO(3)$ admits a metric of non-negative sectional and positive vertizontal curvature. In particular, $\dim B\geq 4$.

	\subsection*{Notation and conventions}
		We denote by $R_{\ga}$ the Riemannian tensor of the metric $\ga$: 
	\[R_{\ga}(X,Y)Z=\nabla_X\nabla_YZ-\nabla_Y\nabla_XZ-\nabla_{[X,Y]}Z,\]
	where $\nabla$ stands for the Levi-Civita connection of $\ga$. We denote either by $K_{\ga}(X,Y)=\ga(R_{\ga}(X,Y)Y,X)$ or by $R_{\ga}(X,Y),$ making it clear in the context, the unreduced sectional curvature of $\ga$. 
	The Ricci tensor of $\ga$ is defined by
	\[\Ricci_{\ga}(X,Y)=\sum_{i=1}^ng(R(e_i,X)Y,e_i), \]
	where $\{e_1,...,e_n\}$ is an orthonormal basis for $\ga$. The associated quadratic form is denoted by $\Ricci_{\ga}(X)=\Ricci_{\ga}(X,X)$.
	
	Whenever we say we have a \emph{Riemannian principal bundle} we mean that the principal bundle is considered with a Riemannian submersion metric.

		\section{(Classical) Cheeger deformations}
	\label{sec:cdef}
	
	We first recall the procedure known as \textit{Cheeger deformations}. Though the main formulae come from classical references, such as \cite{mutterz} and \cite{Muter}, we shall proceed differently in which concerns the presentation. The reason for that is to make more natural our definition of Cheeger deformation on fiber bundles.
	
	Take the product manifold $M\times G$ with the product metric $\ga+t^{-1}Q$, where $G$ acts on $M$ via isometries and $Q$ is a bi-invariant metric on $G$. We therefore see ourselves with two possibilities of free (and commuting) actions:
	
	\begin{equation}\label{eq:modelo}
\begin{xy}\xymatrix{& G\ar@{..}[d]^{\bullet} & \\ G\ar@{..}[r]^{\star} & M\times G\ar[d]^{\pi}\ar[r]^{\pi'} &M\\ &M&}\end{xy}
\end{equation}

In \eqref{eq:modelo} the action $\bullet$ stands to
\begin{equation}\label{eq:bullet}
r\bullet (m,g) := (m,rg),
\end{equation}
while the action $\star$ is nothing but the associated bundle action on $M\times G$, that is
\begin{equation}\label{eq:star}
r\star (m,g) := (rm,rg).
\end{equation}

Therefore, $\pi((m,g)) := m$ meanwhile $\pi'((m,g)) := g^{-1}m$. Since $\pi$ and $\pi'$ define principal bundles, the metric $\textsl g+t^{-1}Q$ induces via $\pi$ and $\pi'$, respectively, the metrics $\ga$ (the original one) and $\textsl g_t$, a \emph{Cheeger deformation} of $\textsl g$. Also note that although $\textsl g_1 = \textsl g$, the horizontal space obtained via $\pi'$ has a different angular position in relation to the horizontal space obtained via $\pi$, what can be directly checked from the fact that
$\textsl g_1(\cdot, \cdot) = \textsl g(C_1\cdot,\cdot) = \textsl g((1+P)^{-1}\cdot,\cdot).$

	Throughout the paper it is shall be denoted by $\mathfrak{m}_p$ the $Q$-orthogonal complement of $\mathfrak{g}_p$, the Lie algebra of $G_p$. We recall that $\mathfrak{m}_p$ is isomorphic to the tangent space to the orbit $Gp$ via \textit{action fields}: for any $U\in \lie g$ the corresponding action field out of $U$ is defined by the rule
	\begin{equation*}
	U^*_p=\frac{d}{dt}\Big|_{t=0}e^{tU}p.
	\end{equation*}
	It is straightforward to check that the map $U\mapsto U^*_p$ is a linear morphism whose kernel is $\lie g_p$. This manner, any vector tangent to $T_pGp$ is said to be \emph{vertical}, hence, such a space is named as the \emph{vertical space} at $p$, being denoted by $\mathcal{V}_p$. For each $p\in M$ its orthogonal complement, denoted by $\cal H_p$, is named \emph{horizontal space}. A tangent vector $\overline X \in T_pM$ can be uniquely decomposed as $\overline X = X + U^{\ast}_p$, where $X$ is horizontal and $U\in \lie m_p$.
	
	To more feasible to understand the geometric properties of Cheeger deformations, next we shall define useful tensors associated with Cheeger deformations, see \cite{mutterz} for further clarifications.
	\begin{enumerate}
		\item The \emph{orbit tensor} at $p$ is the linear map $P : \mathfrak{m}_p \to \mathfrak{m}_p$ defined by
		\[\textsl g(U^{\ast},V^{\ast}) = Q(PU,V),\quad\forall U^{\ast}, V^{\ast} \in \mathcal{V}_p\]
		\item For each $t > 0$ we define $P_t:\lie m_p\to \lie m_p$ as 
		\[\textsl g_t(U^{\ast},V^{\ast}) = Q(P_tU,V), \quad\forall U^{\ast}, V^{\ast} \in \mathcal{V}_p\]
		\item The \emph{metric tensor of $\textsl g_t$}, $C_t:T_pM\to T_pM$ is defined as
		\[\textsl g_t(\overline{X},\overline{Y}) = \textsl g(C_t\overline{X},\overline{Y}), \quad\forall \overline{X}, \overline{Y} \in T_pM\]
	\end{enumerate}
	All the three tensors above are symmetric and positive definite. The next proposition shows how they are related to each other and to the original metric quantities.
	
	\begin{proposition}[Proposition 1.1 in \cite{mutterz}]\label{propauxiliar}The tensors above satisfy:
		\begin{enumerate}
			\item \label{eq:pt} $P_t = (P^{-1} + t1)^{-1} = P(1 + tP)^{-1}$,
			\item If $\overline{X} = X + U^{\ast}$ then $C_t(\overline{X}) = X + ((1 + tP)^{-1}U)^{\ast}$.
		\end{enumerate}
	\end{proposition}
	
 It worth pointing it out that as first observed by Cheeger and playing a vital role in \cite{Muter}, the metric tensor $C_t^{-1}$ can be used to define a very informative reparametrization of $2$-planes to the computation of sectional curvature. Indeed, using this reparametrization we can observe that Cheeger deformations do not create `new' planes with zero sectional curvature, meaning that
	\begin{theorem}\label{thm:curvaturasec}
		Let $\overline{X} = X + U^{\ast},~ \overline{Y} = Y + V^{\ast}$ be tangent vectors. Then $\kappa_t(\overline X,\overline Y) := R_{\textsl g_t}(C_t^{-1}\overline{X},C_t^{-1}\overline{Y},C_t^{-1}\overline{Y},C_t^{-1}\overline{X})$ satisfies
		\begin{equation}\label{eq:curvaturaseccional}
		\kappa_t(\overline X,\overline{Y}) = \kappa_0(\overline{X},\overline{Y}) +\frac{t^3}{4}\|[PU,PV]\|_Q^2 + z_t(\overline{X},\overline{Y}),
		\end{equation}
		where $z_t$ is non-negative.
	\end{theorem}
	We refer to either \cite[Proposition 1.3]{mutterz} or \cite[Lemma 3.5]{cavenaghiesilva} for the details on the proof and more references. Also, with the aim of concluding this section, next we recall a formula for the Ricci curvature of Cheeger deformed metric (see also \cite[Lemma 2.6, p.7)]{cavenaghiesilva}).
	
	\subsection{Ricci curvature}
	
	Let $\{v_1,\ldots,v_k\}$ be a $Q$-orthonormal basis of eigenvectors of $P : \lie m_p \to \lie m_p$, with eigenvalues $\lambda_1\leq\ldots\leq\lambda_k$. Given a $\ga$-orthonormal basis $\{e_{k+1},..,e_n\}$ for $\cal H_p$, we fix the $\ga$-orthonormal basis
	$\{e_1,\ldots,e_k,e_{k+1},\ldots,e_n\}$ for $T_pM$, where $e_i = \lambda_i^{-1/2}v^{\ast}_i$ for $i\leq k$.
	
	The follow claim can be straightforwardly checked:
	\begin{claim}
		The set $\{C_t^{-1/2}e_i\}_{i=1}^n$ is a $\ga_t$-orthonormal basis for $T_pM$. Moreover, $C_t^{-1/2}e_i = (1+t\lambda_i)^{1/2}e_i$ for $i \le k$ and $C_t^{-1/2}e_i = e_i$ for $i > k.$ 
	\end{claim}
	
	Define the horizontal Ricci curvature as
	\begin{equation}\label{eq:riccih}
	\Ricci^{\mathbf h}(\overline{X}) := \sum_{i=k+1}^nR(\overline{X},e_i,e_i,\overline X).
	\end{equation}
	
	\begin{lemma}\label{lem:baseapropriada}For $\{e_1,...,e_n\}$ as above,
		\begin{multline}\label{eq:riccicurvature}
		\Ricci_{\textsl g_t}(\overline{X}) = \Ricci_{\textsl g}^{\mathbf h}(C_t\overline{X}) + \sum_{i=1}^nz_t(C_t^{1/2}e_i,C_t\overline{X}) \\+ \sum_{i=1}^k\frac{1}{1+t\lambda_i}\Big(\kappa_0(\lambda_i^{-1/2}v^{\ast}_i,C_t\overline{X}) + \frac{\lambda_it}{4}\|[v_i,tP(1+tP)^{-1}\overline{X}_\lie g]\|_Q^2\Big).
		\end{multline}
		Moreover,
		\begin{equation}\label{eq:limitricci}
		 \Ricci_{\textsl g_t}(\overline{X}) = \Ricci_{\textsl g}^{\mathbf h}(C_t\overline{X}) + \sum_{i=1}^k\frac{1}{4}\|[v_i,U]\|_Q^2  + \lim_{t\to\infty}\sum_{i=1}^nz_t(C_t\overline X,C_t^{1/2}e_i).
		\end{equation}
		In particular, if the action $G$ is free and $\bar{\textsl g}$ denotes the orbital distance metric in $M/G$ it holds that
		\begin{equation}\label{eq:limitfinalricci}
		 \lim_{t\to\infty}\Ricci_{\textsl g_t}(X) = \Ricci_{\overline{\textsl g}}(d\pi X)
		\end{equation}
	\end{lemma}	
	\begin{proof}
		A straightforward computation following equation \eqref{eq:curvaturaseccional} gives
		\begin{multline*}
		\Ricci_{\textsl g_t}(C_t^{-1}\overline{X}) = \sum_{i=1}^{n}R_{\textsl g_t}(C_t^{-1/2}e_i,C_t^{-1}\overline{X},C_t^{-1}\overline{X},C_t^{-1/2}e_i) =\sum_{i=1}^n\kappa_t(C_t^{1/2}e_i,\overline{X})\\
		= \sum_{i=1}^n\kappa_0(C_t^{1/2}e_i,\overline{X}) + \sum_{i=1}^nz_t(C_t^{1/2}e_i,\overline{X}) + \frac{t^3}{4}\sum_{i=1}^k\|[PC_t^{1/2}\lambda_i^{-1/2}v_i,P\overline{X}_\lie g]\|_Q^2
		\\= \Ricci_{\textsl g}^{\mathbf h}(\overline{X}) + \sum_{i=1}^nz_t(C_t^{1/2}e_i,\overline{X}) + \sum_{i=1}^k\frac{1}{1+t\lambda_i}\Big(\kappa_0(\lambda_i^{-1/2}v^{\ast}_i,\overline{X}) + \frac{\lambda_it^3}{4}\|[v_i,P\overline{X}_\lie g]\|_Q^2 \Big).
		\end{multline*}
		Equations \eqref{eq:riccicurvature}, \eqref{eq:limitricci} now follows by replacing $\overline{X}$ by $C_t\overline{X}$. Finally, equation \eqref{eq:limitfinalricci} is derived from Lemma 4.2 in \cite{cavenaghiesilva}.
	\end{proof}
	
	We finish this section giving a characterization to the $z_t$-term. This shall bring more clarity to the content of Lemma \ref{lem:zezao} ahead. See also \cite[Lemma 3.5, p. 11]{cavenaghiesilva}.
	
	\begin{lemma}\label{lem:zezinho}
	It holds that
		\begin{equation}\label{eq:z_t-equation}
		z_t(\overline{X},\overline{Y}) = 3t\max_{\substack{Z \in \mathfrak{g} \\ |Z|_Q = 1}}\dfrac{\{dw_Z(\overline{X},\overline{Y}) + \frac{t}{2}Q([PU,PV],Z)\}^2}{t\ga(Z^{\ast},Z^{\ast}) + 1},
		\end{equation} 
		where $dw_Z$ is defined as: 
		\begin{align}\label{eq:auxiliaryform}
		w_Z : TM &\to \mathbb{R}\\
		\overline X &\mapsto \textstyle\frac{1}{2}\ga(\overline X,Z^*), \label{eq:auxiliary}
		\end{align}
		where $Z^*$ is the action vector associated to $Z\in \lie g$.
		
		Moreover, if $q\in M^{reg}$, $X,Y\in \cal H_q$ and $U\in\lie g$, then
		\begin{align}
		\label{eq:dw1}dw_Z(U^{\ast},X) &= \frac{1}{2}X\ga(U^*,Z^*)\\
		\label{eq:dw2}dw_Z(X,Y)&=-\frac{1}{2}\ga([X,Y]^{\mathcal{V}},Z^*) = -\ga(A_XY,Z^*).
		\end{align}
		Therefore,
	\begin{equation}\label{eq:zzet}
	 z_t(\overline X,\overline Y) = 3t\left|(1+tP)^{-1/2}P\nabla^{\mathbf{v}}_{\overline X}\overline Y - (1+tP)^{-1/2}t\h[PU,PV]\right|_Q^2.
	\end{equation}
	\end{lemma}
	\begin{proof}
	We begin observing that equations \eqref{eq:dw1} and \eqref{eq:dw2} follow from the definition of exterior derivative of $1$-forms:
	\[d\omega(X,Y) := X\omega(Y) - Y\omega(X) - \omega([X,Y]).\]
	
	To continue, let us make a small digression.
		\begin{claim}\label{claim:z_t}
			Let $ pr : (M,\ga) \to (M/G,\overline \ga)$ be a Riemannian submersion and let $X,Y$ be horizontal vector fields. Then
			\begin{equation*}
			|A_{X}Y|_{\ga}^2 = \max_{Z \in \mathfrak{g}, |Z| = 1}\left\{dw_Z(X,Y)^2{\ga}(Z^*,Z^*)^{-1}\right\}.
			\end{equation*}
		\end{claim}
		\begin{proof}
			This follows from the fact that for any vector space with inner product:
			
			If $V$ is a vector space with inner product  $\langle \cdot,\cdot\rangle$, then
			\begin{equation*}
			\langle v,v\rangle = \max_{x \in V-\{0\}}\langle x,v\rangle^2\langle x,x\rangle^{-1}.
			\end{equation*}
			
			Hence,
			\begin{align*}
			|A_{X}Y|_{\ga}^2 &= \max_{Z \in \mathfrak{g}-\{0\}}\left\{\ga(A_XY,Z^*)^2\ga(Z^*,Z^*)^{-1}\right\},\\
			&= \max_{Z\in \lie{g}, |Z| = 1}\left\{dw_Z(X,Y)^2\ga(Z^*,Z^*)^{-1}\right\}.\qedhere
			\end{align*}
		\end{proof}
		Back to the proof, we now apply Claim \ref{claim:z_t} to the Riemannian submersion $\pi': (M\times G,\ga + \frac{1}{t}Q)\to (M,\ga_t)$ since $z_t(\overline X,\overline Y)$ is precisely $3|A^{\pi'}_{\mathcal L_{\pi'}C_t^{-1}\overline X}\mathcal L_{\pi'}C_t^{-1}\overline Y|_{\ga + t^{-1}Q}^2$. 
		
		Denote by $\overline w_Z^t$, $w_Z$ and $\overline w_Z$ the auxiliary $1$-forms (see equation \eqref{eq:auxiliary}) associated to the actions defined in $(M\times G,\ga + \frac{1}{t}Q)$, $(M,\ga)$, and to the action by left translation in $(G,Q)$, respectively. Note that $\overline w_Z^t= w_Z+t^{-1}\overline w_Z$.
		
		On the one hand, $\mathcal{L}_{\pi'}\overline X$, the horizontal lift of $C_t^{-1}\overline X$ with respect to $\pi' :M\times G\to M$ is given by
		\begin{equation*}
		\mathcal{L}_{\pi'}C_t^{-1}\overline X = (\overline X, -tPU).
		\end{equation*}
		On the other hand $d\overline w_Z(PU,PV) = \frac{1}{2}Q([PU,PV],Z)$ and
		\begin{align*}
		d\overline w_Z^t(\mathcal{L}_{\pi'}C_t^{-1}\overline X,\mathcal{L}_{\pi'}C_t^{-1}\overline Y) &= d\overline w_Z^t((\overline X,-tPU),(\overline Y,-tPV)),\\
		&= dw_Z(\overline X,\overline Y) + \frac{1}{t}d\overline w_Z(-tPU,-tPV),\\
		&= dw_Z(\overline X,\overline Y) + \frac{t}{2}Q([PU,PV],Z).
		\end{align*}
		
		Finally, note that
		\begin{equation*}
		(\ga{+}\frac{1}{t}Q)(Z^*,Z^*) = \ga(Z^*,Z^*) + \frac{1}{t}Q(Z,Z) = \ga(Z^*,Z^*) + \frac{1}{t} = \frac{1}{t}\left(t\ga(Z^*,Z^*) + 1\right)
		\end{equation*}
		once $Q(Z,Z) = 1$. The proof is finished by applying Claim \ref{claim:z_t} to the metric $\ga{+} \frac{1}{t}Q$. Equation \eqref{eq:zzet} follows easily combining equations \eqref{eq:z_t-equation} with \eqref{eq:dw1}, \eqref{eq:dw2}.  \qedhere
	\end{proof}

	\section{Cheeger deformations on associated fiber bundles}
	\label{sec:principal}
	
	In this section we shall use Cheeger deformations to produce deformed metrics on fiber bundles with compact structure groups. As we will describe in section \ref{sec:unified}, such a deformation works as a \textit{canonical model} from which any \textit{basic} vertical metric deformation on fiber bundles shall descend from.
	
	Let $F\hookrightarrow M\stackrel{\pi}{\to}B$ be a fiber bundle from a manifold $M$, with fiber $F$ and compact structure group $G$ and base $B$. We start by recalling that the structure group of a fiber bundle is the group where some choice of transition functions on $M$ takes values. Precisely, if $G$ acts effectively on $F$, then $G$ is a structure group for $\pi$ if there is a choice of local trivializations $\{(U_i,\phi_i:\pi^{-1}(U_i) \to U_i\times F)\}$ 
	such that, for every $i,j$ with $U_i\cap U_j \neq \emptyset$, 
	there is a continuous function $\varphi_{ij} : U_i\cap U_j \to G$ satisfying
	\begin{equation}
	\phi_i\circ\phi_j^{-1}(p,f) = (p,\varphi_{ij}(p)f),
	\end{equation}
	for all $p\in U_i\cap U_j$. 
	
	The existence of $\{\varphi_{ij}\}$ allows us to construct a principal $G$-bundle over $B$ (see \cite[Proposition 5.2]{knI} for details) that we shall denote by $\mathcal{P}$. Furthermore, there exists a principal $G$-bundle $\overline\pi : \mathcal P\times F \to M$ whose principal action is given by
	\begin{equation}\label{eq:action}
	r (p,f) := (rp, rf).
	\end{equation} 
	(For the details see the construction on the proof of \cite[Proposition 2.7.1]{gw}.)
	
	For each pair $\textsl g$ and $\textsl g_F$ of $G$-invariant metrics on $\mathcal P$ and $F$, respectively, there exists a metric $h$ on $M$ induced by $\overline\pi$. 
	Denote by $\mathcal{M}$ the set of all metrics obtained in this way (for instance, if the $G$-action on $F$ is transitive, then every metric on $M$ such that the holonomy acts by isometries on each fiber belongs to $\mathcal{M}$). The set $\mathcal{M}$ is the set of admissible metrics for our deformation: 
	
	\begin{definition}[The deformation]\label{defn} Given $h \in \mathcal{M}$, consider $\textsl{g}{+}\textsl{g}_F$ a product metric on $\cal P\times F$ such that $\overline\pi:(\cal P\times F,\textsl g + \textsl g_F)\to (M,\textsl h)$ is a Riemannian submersion. We define $\textsl h_t$ as the submersion metric induced by $\textsl g_t + \textsl g_F$, where $\textsl g_t$ is the time $t$ Cheeger deformation associated with $\ga$.
	\end{definition}
	
	As it can be seen, the deformation itself is well-defined for a broader class of metrics (for instance, $\cal P\times F$ could have any metric such that each slice $\{p\}\times F$ has  a $G$-invariant metric). However, if the metric is not a product metric, the deformed curvature is harder to control and Theorem \ref{thm:secnew} ahead does not hold.

	\subsection{Curvature formulae}
	
	With the aim of establishing the basic curvature formulae associated to the introduced deformation, we proceed with the following discussion. 
	
	Fix $(p,f) \in \mathcal{P}\times F$. 
	Any $\overline{X}\in T_{(p,f)}(\cal P\times F)$ can be written as $\overline{X} = (X+V^{\vee},X_F+W^*)$, where $X$ is orthogonal to the $G$-orbit on $\mathcal P$, $X_F$ is orthogonal to the $G$-orbit on $F$ and, for $V,W\in\lie g$, $V^{\vee}$ and $W^{\ast}$ are the action vectors relative to the $G$-actions on $\cal P$ and $F$ respectively. Let $P,~P_F$ and $P_t$ be the orbit tensors associated to $\textsl g,~\textsl g_F$ and $\textsl g_t$, respectively. 
	We claim that $\overline X$ is $\textsl g_t + \textsl g_F$-orthogonal to the $G$-orbit of \eqref{eq:action} if and only if
	\begin{equation}\label{eq:horbarpi}
	\overline{X} = (X-(P^{-1}_tP_FW)^{\vee},X_F + W^{\ast}).
	\end{equation}
	for some $W\in\lie m_f$. 
	
	Indeed, a vector $(X+{V^{\vee}},X_F+W^{\ast})$ is horizontal if and only if, for every $U\in \lie g$:
	\begin{multline*}
	0 = \left(\textsl g_t{+}\textsl g_F\right)((X+V^{\vee},X_F+W^{\ast}), (U^{\vee},U^{\ast}))= 
	\textsl g(V^{\vee},U^{\vee}) + \textsl g_F(W^{\ast},U^{\ast})  \\=  Q(P_tV + P_FW,U).
	\end{multline*}
	Since $U$ is arbitrary, we conclude that $V = -P^{-1}_tP_F W$. 
	
	Keeping in mind that the point $(p,f)$ is fixed, we abuse notation and denote
	\begin{equation}\label{eq:conventionU*}
	d\bar\pi_{(p,f)}(X,X_F+U^*) := X+X_F+ U^*.
	\end{equation}
	
	Define the tensors $\tilde P_t,\tilde C_t:\mathfrak{m}_f\to \mathfrak{m}_f$,
	\begin{align}
	\tilde P_t & := P_F(1+P_t^{-1}P_F)^{-1}=(P_F^{-1}+P_t^{-1})^{-1},\\
	\tilde C_t & := -C_tP_t^{-1}\tilde P_t=-P^{-1}\tilde P_t.
	\end{align}	
	
	\begin{claim}\label{claim:lift}
		Let $\cal L_{\overline\pi}: T_{\overline\pi(p,f)}M \to T_{(p,f)}(\mathcal{P}\times F)$ be the horizontal lift associated to $\overline\pi$. Then,
		\begin{equation}\label{eq:lift}
		\cal L_{\overline\pi}(X+X_F+U^*)=(X-(P_t^{-1}\tilde P_tU)^\vee,X_F+(P_F^{-1}\tilde P_tU)^*).	
		\end{equation}
	\end{claim}
	\begin{proof} First observe that the right-hand-side of \eqref{eq:lift} is of the form \eqref{eq:horbarpi}: take $W=P_F^{-1} \tilde P_tU $, so $ P_t^{-1}\tilde P_tU=P^{-1}_tP_FW$. Therefore, it is sufficient to verify that 
		\[d\bar\pi (X-(P_t^{-1}\tilde P_tU)^\vee,X_F+(P_F^{-1}\tilde P_tU)^*)=X+X_F+U^*.\]
	Since $\ker d\bar\pi=\{(U^\vee,U^*)~|~U\in\lie g\}$, convention \eqref{eq:conventionU*} gives $d\bar\pi(U^\vee,0)=-U^*$, thus
	\begin{equation*}
	d\bar\pi (X-(P_t^{-1}\tilde P_tU)^\vee,X_F+(P_F^{-1}\tilde P_tU)^*)=X+X_F+((P_t^{-1}+P_F^{-1})\tilde P_tU)^*=X+X_F+U^*
	\end{equation*}
	since $\tilde P_t=(P_t^{-1}+P_F^{-1})^{-1}$.
	\end{proof}
	Next, we prove that the unreduced sectional curvature is nondecreasing for reparameterized planes (Theorem \ref{thm:secnew}). For the reparametrization, extend  $\tilde C_t$ to $T_{\overline\pi(p,f)}M$ via
	\begin{equation}
	\tilde C_t(X+X_F + U^{\ast}) := X + X_F + (\tilde C_tU)^{\ast}.
	\end{equation} 

 Mimicking M. M\"uter's approach, we obtain a similar result to Theorem \ref{thm:curvaturasec} defining $\tilde \kappa_{t}(\tilde X, \tilde Y) = R_{\textsl h_t}(\tilde C_t^{-1}X,\tilde C_t^{-1}Y,\tilde C_t^{-1}Y,\tilde C_t^{-1}X)$.
	\begin{theorem}[Sectional curvature]\label{thm:secnew}
		Let $\textsl h\in\cal M$ and $\textsl g_t + \textsl g_F$ be as in Definition \ref{defn}. Then, for every pair $\tilde X = X + X_F + U^*$, $\tilde Y = Y + Y_F + V^*$ , 
		\begin{equation*}\label{eq:novaseccional}
		\tilde \kappa_{t}(\tilde X, \tilde Y) = \kappa_t(X{+}U^{\vee},Y{+}V^{\vee}) {+} K_{\textsl g_F}(X_F - (P_F^{-1}PU)^*, Y_F - (P_F^{-1}PV)^*) + \tilde z_t(\tilde X,\tilde Y),
		\end{equation*}
		where $\kappa_t$ is as in Theorem \ref{thm:curvaturasec} and $\tilde z_t$ is non-negative.
	\end{theorem}
	\begin{proof}
		The proof follows from a direct use of  Gray--O'Neill curvature formula and Claim \ref{claim:lift}. Observe that
		\begin{equation*}
		\mathcal{L}_{\overline\pi}(\tilde C_t^{-1}\tilde X) = (C_t^{-1}(X+U^{\vee}),X_F - (P_F^{-1}PU)^{\ast}).
		\end{equation*}
		Let $\tilde z_t$ be three times the norm squared of the integrability tensor of $\overline\pi$ applied to $\mathcal{L}_{\overline\pi}\tilde C_t^{-1}\tilde X,\mathcal{L}_{\overline\pi}\tilde C_t^{-1}\tilde Y$ (see \cite{gw} for details). Then,
		\begin{multline*}
		R_{\textsl h_t}(\tilde C_t^{-1}X,\tilde C_t^{-1}Y,\tilde C_t^{-1}Y,\tilde C_t^{-1}X)=K_{\textsl g_t}(C_t^{-1}(X+U^{\vee}),C_t^{-1}(Y+V^{\vee}))\\+K_{\textsl g_F}(X_F - (P_F^{-1}PU)^{\ast},Y_F - (P_F^{-1}PV)^{\ast})+\tilde z_t(\tilde X,\tilde Y)\\=\kappa_t(X{+}U^{\vee},Y{+}V^{\vee}) {+} K_{\textsl g_F}(X_F - (P_F^{-1}PU)^*, Y_F - (P_F^{-1}PV)^*) + \tilde z_t(\tilde X,\tilde Y).
		\end{multline*}
	\end{proof}
	Since the $\tilde z_t$ shall play some role in the next applications in this paper, we shall study it in more detail. We first claim that it isnon-decreasing with respect to $t$. This is a crucial observation since $\tilde z_0$ is an essential part of the initial curvature: since $\bar\pi:(\cal P\times F,g\times g_F)\to(M,h )$ is chosen to be a Riemannian submersion, $\tilde z_0$ is  the $A$-tensor term on the submersion formula. Or, equivalently, taking $t=0$ in Theorem \ref{thm:secnew},
	\begin{multline*}
	K_{\textsl h}(\tilde X,\tilde Y)=\tilde \kappa_{0}(\tilde X, \tilde Y) \\= \kappa_0(X{+}U^{\vee},Y{+}V^{\vee}) {+} K_{\textsl g_F}(X_F - (P_F^{-1}PU)^*, Y_F - (P_F^{-1}PV)^*) + \tilde z_0(\tilde X,\tilde Y)\\
	= K_{\textsl g}(X{+}U^{\vee},Y{+}V^{\vee}) {+} K_{\textsl g_F}(X_F - (P_F^{-1}PU)^*, Y_F - (P_F^{-1}PV)^*) + \tilde z_0(\tilde X,\tilde Y).
	\end{multline*}
	
It is even possible to furnish a precise description to $\widetilde z_t$. Indeed, although we choose to omit the proof, it can be proved exactly as in \cite[Lemma 3.5, p. 11]{cavenaghiesilva}, or Lemma \ref{lem:zezinho}, that
	\begin{lemma}\label{lem:zezao} Let 
	\begin{align*}
	 w^t_Z &: T\cal P \to \mathbb{R}
	\\
		&X+U^{\vee} \mapsto \textstyle\frac{1}{2}\ga_t(X+ U^{\vee},Z^{\vee}).
		\end{align*}
			\begin{align*}
	 w_Z &: TF \to \mathbb{R}
	\\
		& X_F + U^{\ast} \mapsto \textstyle\frac{1}{2}\ga_F(X_F + U^{\ast},Z^*),
		\end{align*}
		where $Z^*$ is the action vector associated to $Z\in \lie g$. Then it holds that
	\begin{multline*}
	3^{-1}\tilde z_t(\tilde X, \tilde Y) = \\\max_{Z\in \lie g\setminus \{0\}}\dfrac{\left\{dw_Z^t(X+C_t^{-1}U^{\vee},Y+C_t^{-1}V^{\vee})+dw_Z(X_F-(P_F^{-1}PU)^*,Y_F-(P_F^{-1}PV)^*)\right\}^2}{\textsl g_t(Z^{\vee},Z^{\vee})+\textsl g_F(Z^*,Z^*)}.
	\end{multline*}
	\end{lemma}

		\subsection{Regularization via Cheeger deformations}
	
	Let $(M,\textsl g)$ be a compact Riemannian manifold with an effective isometric action by a compact Lie group $G$. If $\ga_t$ denotes a one-parameter family of metrics obtained out from $\ga$ via Cheeger deformations, we can check from the expression to $P_t = P(1+tP)^{-1}$, see equation \eqref{eq:pt}, that as $t\to \infty$ the Riemannian metric $\textsl g_t$ degenerates. 
	
	Taking this observation in account, in \cite{solorzano} Searle, Solórzano and Wilhelm show that for any compact subset of the regular stratum of the $G$-action on $M$, re-scaling the fibers of the fiber bundle $\pi : (M^{reg},\ga_t) \rightarrow M^{reg}/G$, with the same parameter $t$, a procedure known as \textit{Canonical Variation}, implies that for any integer $p\geq 0$ it holds the convergence of ${\widetilde g}_t: = t{\ga_t}|_{\cal V}\oplus \ga|_{\cal H}$, in the $C^p$-topology, to a Riemannian metric with totally geodesic fibers, see \cite[Theorem A]{solorzano}. 
	
	It follows, in particular, that the class $\mathbf{P}^{\Omega}$ of principal bundles with \textit{connection metrics} is invariant by Cheeger deformations. More interesting, any principal bundle with an invariant submersion metric is attracted to $\mathbf{P}^{\Omega}$ by the means of Cheeger deformations.
	
	In this section, we shall prove the analogous result to the class of Riemannian submersions on fiber bundles with compact structure group and total space.
	
	Let $(M,\textsl h_t)\rightarrow (B,\textsl g_B)$ be a complete Riemannian fiber bundle where $\textsl h_t$ is induced via Definition \ref{defn}. We prove that for any compact $K\subset M$ and any non-negative integer $p$, the canonical deformation
	\begin{equation}\label{eq:limitmetric}
	  \widetilde {\textsl h_t}(x) := t\textsl h_t(x)|_{\cal V} + \textsl h_t(x)|_{\cal H},~x\in K
	\end{equation}
	where $(\cal H,\textsl h_t|_{\cal H})$ is isometric to $(TB,\textsl g_B)$, converges to, in the $C^p$-topology, to a Riemannian submersion metric with totally geodesic fibers on $M\rightarrow B$ and horizontal distribution isometric to $\cal H\cong TB$. 
	
	To begin with, let us explain a little bit further about the admissible metrics to the deformation since we have the restriction imposed by the class $\mathcal M$ (recall Definition \ref{defn}).
	
	Suppose that we start with a Riemannian metric $\textsl g_B$ on $B$. Then let $\cal P$ be the principal bundle obtained from $M\rightarrow B$ and let $\omega : TP \rightarrow \lie g$ be any connection there defined. If the structure group of $M\rightarrow B$ has a biinvariant metric $Q$, we impose the \textit{Kaluza--Klein (or connection) metric} on $\cal P$:
	\begin{equation}\label{eq:kaluzaklein}
	 \textsl g := \textsl g_B + Q(\omega,\omega).
	\end{equation}

	Assuming that $G$ is compact one observes that it is always possible to assume that $\omega$ is $G$-invariant. Therefore, any metric $\textsl h$ obtained this way belongs to $\cal M$ for the invariant metrics on $\cal P$ defined by equation \eqref{eq:kaluzaklein}. Theorem \ref{thm:analogousSolorzano} shall show that every metric $h$ in the class $\cal M$ is attracted by the subset of $\cal M$ of metrics obtained in this manner.
	
	Let us first prove that the metric \eqref{eq:limitmetric} shall approach a metric with totally geodesic fibers as $t$ grows large.
	To do so, we first observe that the shape operator $\widetilde S_X$ for any $X\in \cal H$ associated to $\widetilde {\textsl h_t}$ coincides with the shape operator $S^t_X$ of $\textsl h_t$. Indeed, according to equation $(2.1.7)$ in \cite[Chapter 2, p. 47]{gw} the \textit{vertical component} of the covariant derivative ${\widetilde \nabla}_TX$ of a canonical variation remains unchanged provided if $T$ is vertical and $X$ is horizontal. 
	
	 It then suffices to check that $\lim_{t\to\infty}\max_{|X| = 1}|S^t_X|_t = 0$ in $K$. To do so, observe that since the shape operator is symmetric with respect to the metric it is obtained from, it holds that 
	\[\max_{|X|_{\aga}=1}|S^t_X|_{\textsl h_t} = \max_{|X|_{\aga}= |V|_{\aga}=1}|\textsl{h}(-\widetilde C_t{{\nabla_t}^{\mathbf{v}}_VX},V)|.\] Finally, it can be directly checked from the Koszul formula that the right hand side on the above equation goes to $0$ as $t\to \infty$: since
	\begin{align*}
	 \textsl{h}(-\widetilde C_t{{\nabla_t}^{\mathbf{v}}_VX},V) &= X\textsl h(\widetilde C_tV,V) + \textsl h([X,V],\widetilde C_tV) + \textsl h([X,V],\widetilde C_tV)
	\end{align*}
	and $\widetilde C_t = -P^{-1}(P_F^{-1} + P_t^{-1})^{-1}$, $P_t^{-1} = P^{-1}(1+tP)$ it is clear that the only possible problematic term is $ X\textsl h(\widetilde C_tV,V)$. However, applying the Leibiniz rule, one gets that
	\begin{equation*}
	 X\textsl h(\widetilde C_tV,V) = \textsl h(\nabla_X\widetilde C_tV,V) + \textsl h(\widetilde C_tV,\nabla_XV).
	\end{equation*}
	Once more, the Leibiniz rule and the notion of covariant derivative to tensors implies that it suffices to study $\nabla_X\widetilde C_t$. But since
	\begin{align*}
	 \nabla_X\widetilde C_t &= -(P^{-2}\nabla_XP)(P_F^{-1} + P_t^{-1})^{-1} + P^{-1}\nabla_X(P_F^{-1} + P_t^{-1})^{-1},\\
	 \nabla_X(P_F^{-1} + P_t^{-1})^{-1} &= -(P_F^{-1} + P_t^{-1})^{-2}(\nabla_XP_F^{-1} -P_t^{-2}\nabla_XP_t),\\
	 \nabla_XP_t &= \nabla_X(1+tP) = t\nabla_XP.
	\end{align*}
	we are done. $\square$
	
	We prove that:
	
	\begin{theorem}\label{thm:analogousSolorzano}
	Let $\pi : F\hookrightarrow M \rightarrow B$ be a fiber bundle with compact total space and compact structure group $G$. Assume that $\aga$ is a Riemannian submersion metric on $M$ obtained via the submersion $\overline \pi : \cal (P\times F,\textsl g+\ga_F) \rightarrow M$, where $\cal P$ is the associated principal bundle to $\pi$ and $\textsl g, \ga_F$ are, respectively, $G$-invariant metrics on  $\cal P$ and $F$. Then for any integer $p\geq 0$, after an appropriate re-scaling the fibers of $\pi$, the metric deformation $\aga_t$ (Definition \ref{defn}, section \ref{sec:principal}), converges in the $C^p$-topology to a Riemannian submersion metric with totally geodesic fibers. 
	\end{theorem}
	\begin{proof}
	Although we could proceed similarly but independently to Searle--Solórzano--Wilhelm's result (\cite[Theorem A]{solorzano}), we shall use it to both metrics $\textsl g$ in $\cal P$ and $\ga_F$ in $F$. 
	
	Once more, Theorem A in \cite{solorzano} states that if the metric $\widetilde \ga_t$ is a same parameter canonical variation of the Cheeger deformation of $\textsl g$, it follows that for any integer $p\geq 0$ the metric $\widetilde \ga_t$ converges, in the $C^p$ topology, to a metric $\textsl g_{\infty}$ to which the fibers of $\pi$ are totally geodesic. It then suffices to check that a canonical variation of $\textsl g_t$ and of the $t$-Cheeger deformation of $\ga_F$ induces via $\overline \pi : (\cal P + F,\textsl g_t + \textsl g_F) \rightarrow (M,\textsl h_t)$ the metric $\widetilde h_t$.
	
	To do so, observe that since according to equation \eqref{eq:lift} 
	\[\cal L_{\overline\pi}(X+X_F+U^*)=(X-(P_t^{-1}\tilde P_tU)^\vee,X_F+(P_F^{-1}\tilde P_tU)^*)\]
	one has
	\begin{align*}
	 \textsl h_t\left(d\overline \pi\cal L_{\overline\pi}(X+X_F+U^*),\cdot \right) &= \textsl g_t + \textsl g_F\left((X-(P_t^{-1}\tilde P_tU)^\vee,X_F+(P_F^{-1}\tilde P_tU)^*),\cdot\right)\\
	 &= \textsl g_t\left(X-(P_t^{-1}\tilde P_tU)^\vee,\cdot\right) + \textsl g_F\left(X_F+(P_F^{-1}\tilde P_tU)^*,\cdot\right)\\
	 &= \textsl g(X,\cdot) + \textsl g(-C_t(P_t^{-1}\tilde P_tU)^{\vee},\cdot) + \textsl g_F\left(X_F+(P_F^{-1}\tilde P_tU)^*,\cdot\right).
	\end{align*}
	
	The tensor $P_t$ is changed to $tP(1+tP)^{-1}$ under a canonical variation so the following changes hold
	\begin{align*}
	 C_t &= P^{-1}P_t \leftrightarrow t(1+tP)^{-1}\\
	 \tilde P_t &= (P_F^{-1} + P_t^{-1})^{-1} \leftrightarrow (P_F^{-1} + t^{-1}P_t^{-1})^{-1}.
	\end{align*}
	Moreover, the $t$-canonical variation of a $t$-Cheeger deformation of $\textsl g_F$ yields the change
	\begin{align*}
	 P_F \leftrightarrow tP_F(1+tP_F)^{-1}
	\end{align*}

	Hence, making $t\to\infty$ yields
	\begin{multline*}
	 \ga(-(P^{-1}(t^{-1}P_F^{-1}(1+tP_F)+t^{-1}P_t^{-1})^{-1}U)^{\vee},\cdot) \\+ \textsl g_F((t^{-1}P_F^{-1}(1+tP_F)(t^{-1}(1+tP_F)P_F^{-1}+t^{-1}P_t^{-1})^{-1}U)^*,\cdot) \longrightarrow\\ \h\left\{\ga_{F}(U^*,\cdot)-\ga ((P^{-1}U)^{\vee},\cdot)\right\}
	\end{multline*}
	
	On the other hand, since $t\tilde C_t \rightarrow P^{-1}$ as $t\to \infty$ and $d\overline \pi(U^{\vee},0) = -U^{\ast}$ (recall Claim \ref{claim:lift}), the previous computation guarantees the result.
	\end{proof}
	
		\section{Revisiting some classical results: curvature of bi-quotients and almost non-negative curvatures}
	\label{sec:unified}

\subsection{Bi-quotients}
	\label{sec:unifiedbi}
	
	Following \cite[Chapter 2.6]{gw}, let $G$ be a Lie group with a left-invariant metric that is right-invariant under a subgroup $H$. Inspired by constructions of homogeneous spaces, consider the group manifold $G\times H$ which acts isometrically on $G$ via
	\begin{equation}
	 (g,h)\cdot a := gah^{-1}.
	\end{equation}
	
	Since any subgroup $K\leq G\times H$ acts via the same manner on $G$, if this $K$-action happens to be free then the the orbit space $G//K$ is called a \textit{bi-quotient} of $G$. Moreover, once the $K$ action on $G$ is via isometries, there is a metric on $G//K$ that makes the quotient projection $\pi : G\rightarrow G//K$ to be a Riemannian submersion. 
	
	Observe that we can see $G//K$ as the total space of a fiber bundle with trivial fiber, in the sense that the fiber $F$ is a just point and the bundle projection is the identity map. This manner, we can derive the curvature formula to certain deformed metrics on bi-quotients from the previous section.
	
	Indeed, note that if we consider $G\times K$ with the $K$-action defined by \eqref{eq:star} then the projection $\pi'$ recovers $G$ as the respective orbit space. Therefore, composing $\pi'$ with the projection $\pi$ we obtain the corresponding projection $\overline \pi$ (recall Definition \ref{defn})
	\begin{equation}
	 \xymatrix{
	  K\ar@{..}[rr] && G\times K \times \{e\}\ar@{->}[rr]^{\pi'} \ar@/_/@{.>}[rrrr]_{\overline \pi} && \ar@{->}[rr]^{\pi} G \ar@{->}[rr] &&  G//K
	 }
	\end{equation}

The purpose of the computations presented here is the one of showing how the curvature formula on bi-quotients can be very simplified considering the deformation we have introduced in section \ref{sec:principal}.

Indeed, a simple use of Theorem \ref{thm:secnew} implies that the sectional curvature of $M := G//H$ can be read from $\overline\pi$ via the curvature of $\cal P = G$ since $F = \{e\},$ the identity in $K$. Namely,
\begin{theorem}
Let $Q$ be a bi-invariant metric on $K = G\times H$ and $\textsl g$ be a left invariant metric on $G$ which is right-invariant by the elements of the form $\{e\}\times H$. If $\textsl g_t$ is the metric on $G//K$ induced by $\overline \pi : (G\times K,\textsl g + t^{-1}Q) \rightarrow (G//K,\textsl h_t)$, the sectional curvature of $\textsl h_t$ satisfies
\begin{equation}
\kappa_{\textsl h_t}(\tilde X, \tilde Y) = \kappa_{\textsl g_t}(X{+}U^{\vee},Y{+}V^{\vee}) + \widetilde z_t(\tilde X,\tilde Y)
\end{equation}
where $\tilde X = X + X_F + U^*$, $\tilde Y = Y + Y_F + V^*$ and $\tilde z_t$ is computed in Lemma \ref{lem:zezao}.
\end{theorem}

We proceed developing the formulae to the Ricci curvature of the metric deformation given in Definition \ref{defn}, presented in section \ref{sec:principal}. Further applications on bi-quotients shall be also obtained next.

		\subsection{Almost non-negative curvatures and positive Ricci curvature}

	Since we are relying on the horizontal lift $\cal L_{\overline\pi}$ defined by \eqref{eq:lift} to compute curvatures, in order to study the Ricci curvature of $\textsl h_t$ we begin by constructing an appropriate basis for the horizontal space of $\overline\pi$ with respect to $\textsl g_t + \textsl g_F.$	

	Consider a $Q$-orthonormal basis $\{v_k(0)\}$ of $\mathfrak{m}_f$ and define
	\begin{equation*}\label{eq:basemovel}
	v_k(t) = \tilde P_t^{-1/2}v_k(0).
	\end{equation*}
	\begin{lemma} \label{lem:horbasis}
		The set
		\begin{equation}\label{eq:123}
		\{(-{P_t^{-1}\tilde P_tv_k(t)}^{\vee},{P_F^{-1}\tilde P_tv_k(t)}^*)\} = \{(-P_t^{-1}\tilde P_t^{1/2}v_k(0)^\vee, P_F^{-1}\tilde P_t^{1/2}v_k(0)^*)\}
		\end{equation}
		is $\textsl g_t + \textsl g_F$-orthonormal and $\textsl g_t + \textsl g_F$ orthogonal to $(U^\vee,U^*)$, for every  $U\in \lie g$.
	\end{lemma}
	\begin{proof}
	Note that the elements in \eqref{eq:123} are of the form \eqref{eq:lift}. Thus, it is sufficient to show that the set \eqref{eq:123} is orthonormal. A straightforward computation gives:
		\begin{multline*}\label{eq:ortonormalidade}
		(\textsl g_t+ \textsl g_F)((-{P_t^{-1}\tilde P_tv_i(t)}^{\vee},{P_F^{-1}\tilde P_tv_i(t)}^*),(-{P_t^{-1}\tilde P_tv_j(t)}^{\vee},{P_F^{-1}\tilde P_tv_j(t)}^*))\\ 
		= Q(\tilde P_tv_i(t),P_t^{-1}\tilde P_tv_j(t)) + Q(\tilde P_tv_i(t),P_F^{-1}\tilde P_tv_j(t))\\
		= Q(\tilde P_tv_i(t),(P_t^{-1} + P_F^{-1})\tilde P_tv_j(t))\\
		= Q(\tilde P_tv_i(t),v_j(t)) = Q(\tilde P_t^{1/2}v_i(0), \tilde P_t^{-1/2}v_j(0)) = \delta_{ij},
		\end{multline*}
		where we have used that $(P_t^{-1} + P_F^{-1}) = \tilde P_t^{-1}$ and that $\tilde P_t$ is symmetric.
	\end{proof}
	
	Let $\{e_i^{B}\}$ and $\{e_j^F\}$ be orthonormal bases  for the spaces normal to the orbits on $\cal P$ and on $F$, respectively. We complete the set on Lemma \ref{lem:horbasis} to a $\textsl g_t + \textsl g_F$-othornormal basis for the $\bar{\pi}$-horizontal space:
	\begin{equation}\mathcal{B}_t := \label{eq:htbasis}\left\{(e_i^B,0) ,(-{P_t^{-1}\tilde P_t^{1/2}v_k(0)}^{\vee},{P_F^{-1}\tilde P_t^{1/2}v_k(0)}^*), (0,e_j^F)\right\}.\end{equation}
	Denote by $e_1,\ldots,e_n$ the elements in $\cal B_t$.
%
	
	\begin{lemma}\label{prop:definitivo}For any $(p,f)\in\cal P\times F$ and $X+X_F+U^*\in T_{\bar{\pi}(p,f)}M$,
		\begin{equation}\label{eq:asymptotic}
		\lim_{t\to \infty}\Ricci_{\textsl h_t}(X + X_F + U^*) \ge \Ricci_{\textsl g}^{\mathbf{h}}(X) + \Ricci_{\textsl g_F}^{\mathbf{h}}(X_F) + \sum_k\frac{1}{4}\|[v_k(0),U]\|_Q^2.
		\end{equation}
	\end{lemma}
		\begin{proof}
		Using the basis $\cal B_t$, from \eqref{eq:htbasis}, and Theorem \ref{thm:secnew}, we have:
\begin{multline*}
\Ricci_{\textsl h_t}(\tilde X)=\sum_{i=1}^n\kappa_t(\tilde C_t\tilde X,\tilde C_t e_i)\geq \\\sum_i \kappa_t(X-(C_tP_t^{-1}\tilde P_tU)^{\vee},e_i^B) + \sum_{k}\kappa_t(X - (C_tP_t^{-1}\tilde P_tU)^{\vee},-C_tP_t^{-1}\tilde P_t^{1/2}v_k(0)^{\vee})\\ + \sum_{j}K_{\textsl g_F}(X_F + (P_F^{-1}\tilde P_tU)^*,e^F_j) + \sum_{k}K_{\textsl g_F}(X_F+(P_F^{-1}\tilde P_tU)^*,P_F^{-1}\tilde P_t^{1/2}v_k(0)^*).
			\end{multline*}
			On the other hand, the  $\tilde P_t$ satisfies:
						\begin{align}
							\lim_{t\to \infty}t\tilde P_t &= 1,\label{claim:vaicomtudo}\\
							\underset{t\to \infty}{\lim}P_t^{-1}\tilde P_t &= 1.\label{claim:vaicomtudo2}
						\end{align}
			In particular, $\tilde P_t \to 0$ as $t\to\infty$. Equation \eqref{claim:vaicomtudo} follows since $t\tilde P_t = P_F(P_t+P_F)^{-1}tP_t$, $P_t \to 0$ and $tP_t \to 1$. Equation \eqref{claim:vaicomtudo2} follows since 
\begin{equation*}
			\lim_{t\to\infty}P_t^{-1}\tilde P_t=\lim_{t\to\infty}(tP_t)^{-1}\lim_{t\to\infty}t\tilde P_t=1.
\end{equation*}						
			Using \eqref{claim:vaicomtudo}, we observe that
			\begin{align}
			\lim_{t\to \infty}\left\{\sum_{j}K_{\textsl g_F}(X_F + (P_F^{-1}\tilde P_tU)^*,e^F_j) + \sum_{k}K_{\textsl g_F}(X_F+(P_F^{-1}\tilde P_tU)^*,P_F^{-1}\tilde P_t^{1/2}v_k(0)^*)\right\} \nonumber\\=\Ricci_{\textsl g_F}^{\mathbf h}(X_F). \label{eq:1}
			\end{align}
			
			Moreover, using equation \eqref{eq:curvaturaseccional} and that $C_tP_t^{-1} = P^{-1}$, 
			\begin{equation}\label{eq:2}
			\lim_{t\to \infty}\sum_i \kappa_t(X-(C_tP_t^{-1}\tilde P_tU)^{\vee},e_i) \ge \lim_{t\to \infty}\Ricci_{\textsl g}^{\mathbf h}(X-(P^{-1}\tilde P_tU)^{\vee})=\Ricci_{\textsl g}^{\mathbf h}(X).
			\end{equation}
			For the remaining term:
			\begin{multline*}
			\kappa_t(C_tX - (C_tP_t^{-1}\tilde P_tU)^{\vee},-C_tP_t^{-1}\tilde P_t^{1/2}v_k(0)^{\vee}) \\\ge
			K_{\textsl g}(X-(P^{-1}\tilde P_tU)^{\vee},-P^{-1}\tilde P_t^{1/2}v_k(0)^{\vee}) +
			\frac{t^3}{4}\|[\tilde P_tU,\tilde P_t^{1/2}v_k(0)]\|^2_Q
			\end{multline*}
			Using \eqref{claim:vaicomtudo} and \eqref{claim:vaicomtudo2}, we obtain
			\begin{equation}\label{eq:3}
			\lim_{t\to \infty}\kappa_t(C_tX - (C_tP_t^{-1}\tilde P_tU)^{\vee},-C_tP_t^{-1}\tilde P_t^{1/2}v_k(0)^{\vee}) \ge \frac{1}{4}\|[v_k(0),U]\|_Q^2.
			\end{equation}
			Lemma \ref{prop:definitivo} follows by putting together \eqref{eq:1}, \eqref{eq:2} and \eqref{eq:3}.
		\end{proof}
		
	The following needed lemmas can be readily obtained via straightforward computations, therefore, we chose to not include their proofs.
		\begin{lemma}
		It holds that
		\begin{equation}
		\lim_{t\to\infty}\tilde z_t(\tilde C_t\tilde X,\tilde C_t\tilde Y) = 3\max_{Z\in \lie g\setminus\{0\}}\dfrac{\lim_{t\to \infty}\left\{dw^t_Z(X-U^{\vee},Y-V^{\vee}) + dw_Z(X_F,Y_F)\right\}^2}{\textsl g_F(Z^*,Z^*)}.
		\end{equation}
		Moreover, if the orbits on $\cal P$ are totally geodesic then 
		\begin{equation}
		 \lim_{t\to\infty}dw^t_Z(X-U^{\vee},Y)^2 = \textsl g(A_XU,Z)^2.
		\end{equation}
		\end{lemma}
		
		\begin{lemma}\label{lem:limitecompleto}
		If $\{e_i\}$ denotes the set of elements in the basis $\cal B_t$ (see equation \eqref{eq:htbasis}) and the $G$ orbits on $\cal P$ are totally geodesic then for any $\widetilde X = X + X_F + U^*$ it holds that
		\begin{equation}
		 \lim_{t\to\infty}\sum_i\tilde z_t(\tilde C_t d\overline\pi e_i, \tilde C_t\tilde X) = 3\sum_{i=1}|A^{\pi}_Xe^B_i|^2_{\textsl g} + 3\sum_{j=1}|A^{\pi_F}_{X_F}e_j^F|_{g_F}^2.
		\end{equation}
		
		Therefore,
		\begin{equation}
		 	\lim_{t\to \infty}\Ricci_{\textsl h_t}(X + X_F + U^*) = \Ricci_{\bar{\textsl g}}(d\pi X) + \Ricci^{\mathbf{h}}(X_F) + 3\sum_{j=1}|A^{\pi_F}_{X_F}e_j^F|_{g_F}^2 + \sum_k\frac{1}{4}\|[v_k(0),U]\|_Q^2.
		\end{equation}
		\end{lemma}
		
We now recover the results of L.J. Schwachh{\"o}fer and W. Tuschmann (\cite{schwachhofer2004metrics}) relating the geometry and the topology of bi-quotients. The method applied provides a huge simplification over their work. We re-inforce, however, that the possibility of a simplification was already expected, as observed B. Wilking and W. Ziller, see \cite{mutterz}. More precisely, it is always possible by the means of a Cheeger deformation to prove the existence of metrics of almost non-negative sectional curvature on cohomogeneity one manifolds: the idea consists of putting metrics on non-negative sectional curvature near the singular orbits (recall for instance that these are homogeneous disk-bundles (\cite{gz})) and extend these arbitrarily in the middle. Then it can be shown that all curvatures in the middle go to $0$.

In what follows the method is very different, but it relies in the same spirit: Cheeger deformations tend to shrink the curvature along the orbits. More importantly, they also work as a regularization process: the metrics naturally converge to a metric with totally geodesic fibers. This phenomenon is manifested here in the form that the Ricci curvature of a bi-quotient is completely determined by the Ricci curvature of the upstairs Lie group with a bi-invariant metric.

\begin{theorem}[Schwachh{\"o}fer--Tuschmann]\label{thm:biq}
A bi-quotient $G//K$ of a compact connected Lie group $G$ carries a metric of positive Ricci curvature if, and only if, its fundamental group is finite.
\end{theorem}
\begin{proof}
Recall that according to section \ref{sec:unifiedbi}, $G//K$ can be seen as the total space with trivial fiber $F = \{e\}$ and structure group $K$. Since the tangent space to the `manifold point' $\{e\}$ is only the zero vector it follows that $\Ricci_{\textsl g_F}^{\mathbf h} \equiv 0$. Moreover, denote by $\pi : K\hookrightarrow G\rightarrow G//K$ the Riemannian submersion obtained from the principal bundle with total space $\cal P = G$. According to Lemma \ref{lem:limitecompleto}, if $\overline \ga$ denotes the submersion metric induced by $\pi$ from $\textsl g$ one gets
\begin{align*}
\lim_{t\to \infty}\Ricci_{\textsl h_t}(X + X_F + U^*) &= \Ricci_{\bar{\textsl g}}(d\pi X) + 3\sum_{j=1}|A^{\pi_F}_{X_F}e_j^F|_{g_F}^2 + \frac{1}{4}\sum_k\|[v_k(0),U]\|_Q^2\\
&=\Ricci_{\bar{\textsl g}}(d\pi X) + \frac{1}{4}\sum_k\|[v_k(0),U]\|_Q^2,
\end{align*}
where the last equality comes from the fact that the horizontal space associated to the action of $K$ in $F = \{e\}$ is only the zero vector, i.e, $X_F = 0$. Now the proof is finished by noticing that if $|\pi_1(G//K)| < \infty$ then the same holds for $G$. 
Therefore, since
\begin{align*}
\Ricci_{\bar{\textsl g}}(d\pi X) &= \Ricci^{\mathbf h}(X) + 3\sum_{i=1}^{\dim G//K}|A^{\pi}_Xe_i|^2,
\end{align*}
where $\{e_1,\ldots,e_{\dim G//K}\}$ is an orthonormal basis to the horizontal space of the $K$-action on $G$, associated to the submersion $\pi$, if $Q$ is any bi-invariant metric on $G$ it holds that
\begin{align*}
\Ricci_{\bar{\textsl g}}(d\pi X) &= \Ricci^{\mathbf h}(X) + 3\sum_{i=1}^{\dim G//K}|A^{\pi}_Xe_i|^2,\\
&= \sum_{i=1}^{\dim G//K}\|[X,e_i]\|^2_Q + \tfrac{3}{4}\sum_{i=1}^{\dim G//K}\|[X,e_i]^{\lie k}\|_Q^2,
\end{align*}
where $\lie k$ is the Lie algebra of $K$.
Hence,
\begin{equation*}
\lim_{t\to \infty}\Ricci_{\textsl h_t}(X + X_F + U^*) = \sum_{i=1}^{\dim G//K}\|[X,e_i]^{\lie m}\|^2_Q + \tfrac{7}{4}\sum_{i=1}^{\dim G//K}\|[X,e_i]^{\lie k}\|_Q^2 + \frac{1}{4}\sum_k\|[v_k(0),U]\|_Q^2
\end{equation*}
and so $\lim_{t\to \infty}\Ricci_{\textsl h_t}(X + X_F + U^*) > 0$ since $G$ has finite fundamental the sums of the three kind of brackets cannot vanish simultaneously. 
\end{proof}

We finish providing a result about almost non-negative sectional curvature and positive Ricci curvature, simultaneously, to biquotients. This shall be done taking advantage of the following similar result to \cite[Theorem 0.18]{fukaya-yamaguchi}, which has a very simple proof. Once more, this proof was already known to be possible by B. Wilking and W. Ziller in the context of principal fiber bundles. Here we extent it naturally to general fiber bundles with compact structure group.

\begin{theorem}[Fukaya--Yamaguchi type result]\label{thm:main}
Let $F \hookrightarrow M \to B$ be a bundle with compact structure group $G$, fiber $F$ and base $B$. Assume that $M$ is an associate bundle to $\pi : (\cal P,\textsl g) \to B$ such that:
\begin{enumerate}
\item $K_{\textsl g}\geq 0$;
\item $F$ has a $G$-invariant metric $\textsl g_F$ of non-negative sectional curvature.
\end{enumerate}
Then $M$ admits a sequence of Riemannian metrics
$\{\textsl g_n\}$ such that $\mathrm{sec}_{\textsl g_n} \geq -\frac{1}{n},$ $\mathrm{diam}~(M,\textsl g_n) \leq \frac{1}{n}.$
\end{theorem}

\begin{definition}\label{def:almostnonnegativesec}
A compact manifold $M$ with a family of metrics $(\textsl g_n)$ as on the thesis of Theorem \ref{thm:main} is said to admit \textit{almost non-negative sectional curvature}.
\end{definition}

Theorem \ref{thm:main} is a straightforward consequence of the following lemma.
\begin{lemma}
Let $F\hookrightarrow M \to B$ and $\pi : \cal P \to B$ as on the hypotheses of Theorem \ref{thm:main}. Then for each $\epsilon > 0$ there exists $t_{\epsilon} > 0$ such that for every $t > t_{\epsilon},$ $\tilde \kappa_t(\widetilde X,\widetilde Y) \geq -\epsilon,~|\widetilde X| = |\widetilde Y| = 1.$
\end{lemma}
\begin{proof}
Assume by contradiction that there is $\epsilon > 0$, a sequence $\{t_n\}\nearrow +\infty$ and a sequence of planes $\{\widetilde X_n,\widetilde Y_n\}$ with $|\widetilde X_n| = |\widetilde Y_n| = 1$ satisfying
\begin{equation}\label{eq:contradiction}
\tilde \kappa_n(\widetilde X_n,\widetilde Y_n) \leq -\epsilon.
\end{equation}
By compactness, passing to a subsequence if necessary one extracts a limit plane $\{\widetilde X,\widetilde Y\}$ such that
\begin{equation*}
-\epsilon \geq \lim_{n\to\infty}\left\{\kappa_n(X_n{+}U_n^{\vee},Y_n{+}V_n^{\vee}) {+} K_{\textsl g_F}((X_F)_n - (P_F^{-1}PU_n)^*, (Y_F)_n - (P_F^{-1}PV_n)^*)\right\}.
\end{equation*}

Theorem \ref{thm:curvaturasec} then implies that

\begin{equation}
-\epsilon \geq \lim_{n\to\infty}\kappa_n(X_n{+}U_n^{\vee},Y_n{+}V_n^{\vee}) \geq\lim_{n\to\infty}\left\{\kappa_0(X_n+U_n^{\vee},Y_n + V_n^{\vee} + n^3|[U_n,V_n]|_Q^2\right\}.
\end{equation}
and hence, $\lim_{n\to\infty}[U_n,V_n] = 0$. Therefore,
\begin{equation}
-\epsilon \geq \kappa_0(X+U^{\vee},Y{+}V^{\vee}) {+} K_{\textsl g_F}(X_F-(P_F^{-1}PU)^*,Y_F - (P_F^{-1}PV)^*) \geq 0.
\end{equation}\qedhere
\end{proof}
We thus conclude:
\begin{theorem}[Schwachh{\"o}fer--Tuschmann]\label{thm:biqn}
Any bi-quotient $G//K$ from a compact Lie group $G$ admits a metric with positive Ricci curvature and almost non-negative sectional curvature simultaneously if, and only if, $G//K$ has finite fundamental group.
\end{theorem}

\begin{remark}\label{rem:fat}
Since $\cal P$ is a principal bundle, we could try to impose more rigid hypotheses to produce a metric of non-negative sectional curvature, but this is not possible only via the above method. 
Indeed, even assuming that $G = S^3, SO(3)$ and that $\textsl g_F$ and $\textsl g$ have positive sectional curvature we would have that \[\kappa_0(X,Y{+}V^{\vee}) {+} K_{\textsl g_F}(X_F,Y_F - (P_F^{-1}PV)^*) = 0\] for planes $X = 0,U = 0, V = 0, Y_F = 0.$ In this case,
\[\widetilde X = X_F,~\widetilde Y = Y,\]
so we get no contradiction for any $\epsilon > 0$. 
\end{remark}
A compact manifold with a family of metrics as in the Definition \ref{def:almostnonnegativesec} with $\mathrm{sec}$ changed to $\Ricci$ is a manifold with \textit{almost non-negative Ricci curvature}. As a last result in this section we prove: 

\begin{theorem}\label{thm:previagromov}
Let $F\hookrightarrow M \rightarrow B$ be a fiber bundle with compact structure group $G$ and total space $M$. Also assume that $F$ carries a metric $\textsl g_F$ of non-negative Ricci curvature and $B$ carries a metric $\textsl g_{\epsilon}$ with $\Ricci(\textsl g_{\epsilon}) \geq -\epsilon^2$. Then $M$ carries a metric $\textsl h_{\epsilon}$ with $\Ricci(\textsl h_{\epsilon}) \geq -\epsilon^2$.
\end{theorem}
\begin{proof}
Let $\textsl h_{\epsilon}$ be the connection Riemannian metric on $M$. That is, it has totally geodesic fibers and make $(M,\textsl h) \rightarrow (B,\textsl g_{\epsilon})$ to be a Riemannian submersion. Considering the $\textsl h^{\epsilon}_t$ deformation given by Definition \ref{defn} we see that for large $t$ it holds that, according to Lemma \ref{lem:limitecompleto}, the limit behavior of $\Ricci(h^{\epsilon}_t)$ is 
\begin{align}
\Ricci_{\bar{\textsl g}}(d\pi X) + \Ricci^{\mathbf{h}}(X_F) + 3\sum_{j=1}|A^{\pi_F}_{X_F}e_j^F|_{g_F}^2 + \sum_k\frac{1}{4}\|[v_k(0),U]\|_Q^2 \geq -\epsilon^2.
\end{align}\qedhere
\end{proof}

\section{Some comments on the Petersen--Wilhelm fiber dimension conjecture}
\label{sec:petwill}

Related to the Petersen--Wilhelm conjecture, assume that $S^3, SO(3) \hookrightarrow \cal P \rightarrow B$ is a principal bundle with positive sectional curvature. It is then conjectured that $\dim B > 3$. We conjecture further:
\begin{conjecture}[Principal bundle Strong Petersen--Wilhelm conjecture]\label{conj:principal}
Any $S^3, SO(3)$ principal bundle over a positively curved manifold admits a metric with positive sectional curvature if, and only if, such a submersion is fat.
\end{conjecture}

Assume for instance the validity of Conjecture \ref{conj:principal} and take a fat principal bundle $S^3, SO(3) \hookrightarrow \cal P \rightarrow B$. Regard it with a metric of positive sectional curvature. Now let $(F,\ga_F)$ be a Riemannian manifold with a $S^3, SO(3)$ isometric action. Assume that $\ga_F$ has positive sectional curvature. Then Theorem \ref{thm:secnew} implies that the $\aga_1$ metric deformation (Definition \ref{defn}) has non-negative sectional curvature. More importantly, Remark \ref{rem:fat} implies that the existence of flat planes is \emph{intrinsic} in the sense it only depends on the $G = S^3, SO(3)$ actions on both $\cal P$ and $F$.

Now recall that if $F = S^2$ and $\cal P \rightarrow B$ is the $SO(3)$ principal bundle associated to it, then $\overline \pi: S^2 \hookrightarrow M \rightarrow B$ is fat if, and only if, $\pi : SO(3)\hookrightarrow \cal P \rightarrow B$ is fat (\cite[Proposition 2.22, p.16]{Ziller_fatnessrevisited}). This implies that $\dim B \geq 4$ and hence, Petersen--Wilhelm conjecture is verified in this case. More drastically, the existence of a metric of non-negative sectional curvature on $M$ already verifies the conjecture.

That all said and also taking in account the results in \cite{speranca2017on}, we are tempted to conjecture the following:
\begin{conjecture}
Let $\overline \pi : F\hookrightarrow M \rightarrow B$ a fiber bundle with structure group $S^3$ or $SO(3)$ over a positively curved manifold $B$. If the principal bundle $\pi : S^3, SO(3) \hookrightarrow \cal P \rightarrow B$ associated to it is fat and $\overline \pi$ has a metric of positive vertizontal curvature then $\overline \pi$ is fat.
\end{conjecture}
Let us verify that this is precisely the case to $S^2$-fat bundles, therefore agreeing with our conjecture. More precisely, let us show that $\aga_1$ has positive vertizontal curvature.

Observe that the vertical space associated to $\overline \pi$ consists of vectors tangent to $F = S^2$. Since the $SO(3)$ action on $S^2$ is transitive and $S^2$ can be identified with the homogeneous space $SO(3)/SO(2)$ then for the fixed origin $o\in F$ we have $T_0F \cong \lie so(3)\ominus\lie so(2)$, meaning that $T_oF$ is isomorphic to the complement of $\lie so(2)$ in $\lie so(3)$. Fixing an \text{Ad}-invariant inner product in $\lie so(3)$ such a complement can be chosen to be orthogonal.

Finally, the horizontal space on $\cal P$ at any $p\in \cal P$ is isomorphic to the tangent space $T_{\pi(p)}B$ and it is also isomorphic to the horizontal space orthogonal to $T_oF$ with respect to $\aga_1$. That is,
\begin{equation*}
T_{[(p,o)]}M \cong T_oF\oplus \cal H^{\pi} \cong (\lie so(3)\ominus \lie so(2))\oplus (\lie so(3))^{\perp_{\ga}}.
\end{equation*}
Therefore, any vertizontal plane tangent to $M$ is of the form $U^*\wedge X$ for $U\in \lie so(3)\ominus \lie so(2)$ and $X\in (\lie so(3))^{\perp_{\ga}}$. Therefore Theorem \ref{thm:secnew} implies that
\[\widetilde K_1(X,U^*) \geq K_{\ga_1}(X,U^{\vee}) > 0 \Leftrightarrow U^{\vee} \neq 0.\]
Since the $SO(3)$ action on $\cal P$ is free we have concluded the result.
$\square$

\section*{Acknowledgments}
Part of this work comes from the Ph.D. thesis of the first named author, L. F. Cavenaghi, supported both by S\~ao Paulo Research Foundation FAPESP grant 2017/24680-1 and by CAPES. Also, some ideas of this work were conceived during his former postdoc position at the University of Fribourg, supported in part by the SNSF-Project 200020E\_193062 and the DFG-Priority programme SPP 2026. Currently, L. F. Cavenaghi is supported by CAPES, a grant obtained for the reason of having won the CAPES thesis prize in 2021.
L. Grama is partially supported by S\~ao Paulo Research Foundation FAPESP grants 2018/13481-0, 2021/04003-0, 2021/04065-6 and CNPq grant no. 305036/2019-0.
	
	\bibliographystyle{alpha}
	
	\bibliography{main}

	\end{document}